\documentclass[letterpaper, 10 pt, conference]{ieeeconf}  % Comment this line out if you need a4paper

\IEEEoverridecommandlockouts                              % This command is only needed if 
\usepackage{amsmath,cite}
\usepackage{amssymb}
\usepackage{amsfonts}
\usepackage{mathtools}
\usepackage{color}
\usepackage{bm}
\usepackage{amsthm}
   
    \newtheorem{prop}{Proposition}
     \newtheorem{lem}{Lemma}
    \newtheorem{thm}{Theorem}
    \newtheorem{cor}{Corollary}
    \newtheorem{rem}{Remark}
   
\allowdisplaybreaks

\title{\LARGE \bf
 Low-rank approximated Kalman--Bucy filters using Oja's principal
 component flow for linear time-invariant systems 
}

\author{Daiki Tsuzuki$^{1}$ and Kentaro Ohki$^{1}$% <-this % stops a space
\thanks{*This work was supported by JSPS KAKENHI Grant Numbers JP19K03619, JP21K12097, and JP23K26126. }% <-this % stops a space
\thanks{$^{1}$Department of Applied Mathematics and Physics, Graduate School of Informatics, Kyoto University, Yoshida-Honmachi, Sakyo-ku, Kyoto 606-8501, Japan.
        {\tt\small ohki@bode.amp.i.kyoto-u.ac.jp}}%
}

\begin{document}

\maketitle
\thispagestyle{empty}
\pagestyle{empty}

%%%%%%%%%%%%%%%%%%%%%%%%%%%%%%%%%%%%%%%%%%%%%%%%%%%%%%%%%%%%%%%%%%%%%%%%%%%%%%%%
\begin{abstract}
The Kalman--Bucy filter is extensively utilized across various applications. However, its computational complexity increases significantly in large-scale systems. To mitigate this challenge, a low-rank approximated Kalman--Bucy filter was proposed, comprising Oja's principal component flow and a low-dimensional Riccati differential equation. Previously, the estimation error was confirmed solely for linear time-invariant systems with a symmetric system matrix. This study extends the application by eliminating the constraint on the symmetricity of the system matrix and describes the equilibrium points of the Oja flow along with their stability for general matrices. In addition, the domain of attraction for a set of stable equilibrium points is estimated. Based on these findings, we demonstrate that the low-rank approximated Kalman--Bucy filter with a suitable rank maintains a bounded estimation error covariance matrix if the system is controllable and observable.
\end{abstract}
%%%%%%%%%%%%%%%%%%%%%%%%%%%%%%%%%%%%%%%%%%%%%%%%%%%%%%%%%%%%%%%%%%%%%%%%%%%%%%%%

\section{Introduction}

The Kalman/Kalman--Bucy filter \cite{kalman1960new,kalman1961new} is extensively employed in various disciplines such as aerospace engineering \cite{teixeira2008spacecraft,pittelkau2001kalman}, weather forecasting \cite{cassola2012wind,kunii2012estimating,che2016wind}, and quality management \cite{senel2007kalman,lai2019iot}, for recursive state estimation from noisy and partial observations.
However, the computational demands of the Riccati equation limit the usage of the filter in high-dimensional systems. Challenges persist despite various approaches including model reduction \cite{gugercin2004survey,SHINGIN2018143} and approximations of Riccati differential equations \cite{kirsten2020order,stillfjord2018adaptive,bonnabel2012geometry,yamada2021new}.

Bonnabel et al. \cite{bonnabel2012geometry} introduced a low-rank approximated Kalman--Bucy (LRKB) filter, integrating Oja’s principal component flow (Oja flow) \cite{Oja1982SimplifiedNM} and low-dimensional Riccati differential equations. 
This filter employs a low-rank approximated covariance matrix to decrease the computational load of the Riccati equation. Although it extends to linear time-varying and nonlinear systems as an extended Kalman--Bucy filter, no conditions for bounded estimation error were specified. Furthermore, under certain conditions, the estimation error may significantly worsen.

Yamada et al. \cite{yamada2021new} subsequently refined the low-rank filter to ensure a bounded estimation error if a linear time-invariant (LTI) system is detectable and features a symmetric system $A$ matrix. {Although Yamada et al. \cite{yamada2021comparison} numerically verified that the proposed filter also maintains bounded estimation errors for a non-symmetric $A$ matrix under detectable conditions, a theoretical assurance remains absent.} The focus on symmetric matrices emerges from the analysis and utilization of the Oja flow to isolate the principal or minor components, which focus on positive definite matrices \cite{OJA198569,xu1993least,CHEN20011377,Yoshizawa2001ConvergenceAF,AbsilMahonySepulchre+2008}. {As the equilibrium points and convergence properties of the Oja flow are unresolved for the general $A$ matrix \cite{liu2022convergence, yoshizawa2023dynamical}, an analysis of the Oja flow is crucial for ensuring bounded estimate errors using the modified LRKB filter in practice.}

In this study, we eliminate the restriction on the symmetric system $A$ matrix to verify that the modified LRKB filter proposed in \cite{yamada2021new} demonstrates satisfactory performance. The contributions of this study are stated as follows:

\begin{enumerate}
\item We elucidate the explicit equilibrium points of the Oja flow for general real square matrices (Prop. \ref{prop:equili_str}) and their local stability (Theorem \ref{thm:equili_localconv}). Although disjoint equilibrium sets exist, only one set is stable. We also establish a sufficient condition for the domain of attraction of this stable equilibrium set (Theorem \ref{thm:equili_globalconv}).

\item The conditions for the bounded error covariance matrix of the modified LRKB filter are derived (Theorem \ref{thm:lrkf_rankcondi}).

\end{enumerate}

The remainder of this paper is structured as follows. The Kalman--Bucy and LRKB filters are reviewed in Section \ref{sec:previous}. Thereafter, our contributions to the analysis of the Oja flow along with the equilibrium points and their stability are presented in Section \ref{subsec:Oja equi}, and the domain of attraction for the stable equilibrium points are derived. Based on these results, the conditions for the estimation error covariance matrix to remain finite are outlined in Section \ref{subsec:bounded LKBF}. The conclusions of this study are presented in Section \ref{sec:Conclu}.

\paragraph*{Notation}
The sets of real and complex numbers are denoted by $\mathbb{R}$ and $\mathbb{C}$, respectively. The sets of $n \times m $ real and complex matrices are described by $\mathbb{R}^{n\times m}$ and $\mathbb{C}^{n \times m}$, respectively. The $n\times n$ identity matrix is denoted by $I_{n}$ and $n\times m$ zero matrix is represented by $O_{n,m}$. For $n\times m$ matrix $A$, $A^{\top}$ and $A^{\dagger}$ indicate the transposed and Hermitian conjugates of $A$, respectively. For a real symmetric matrix $A$, $A>0 (A\geq0)$ indicates that $A$ is positive (semi-)definite. For a square matrix $A \in \mathbb{C}^{n \times n}$, its eigenvalues are denoted by $\lambda_{i}(A)\quad(\mathrm{Re}(\lambda_1(A)) \ge \dots \ge \mathrm{Re}(\lambda_n(A)))$, and the corresponding eigenvectors with norm 1 are denoted by $\bm{\psi}_{i}(A)$, $i=1,\dots, n$. If the argument holds, we can simply express $\lambda _i$ and $\bm{\psi}_i$. For the degenerated eigenvalues {$\lambda _{i}$}, we utilized {$\bm{\psi} _{i}$ as generalized eigenvectors}. The Stiefel manifold is denoted as $\mathrm{St}(r,n) := \{X \in {\mathbb{R}^{n \times r}}|{X}^{\top}X=I_r\}$.

%%%%%%%%%%%%%%%%%%%%%%%%%%%%%%%%%%%%%%%%%%%%%%%%%%
\section{Preliminary: Kalman--Bucy filter and its low-rank approximation}\label{sec:previous}

\subsection{Kalman--Bucy filter}\label{subsec:KBF}

As stated in Introduction, the LRKB filter can be employed in linear time-varying or nonlinear systems. However, herein, we focus on its application only on LTI systems. 
\begin{align}
\frac{dx(t)}{dt}=&Ax(t)+Gw(t) \label{eq_x}, \\
y(t)=&Cx(t)+Hv(t)\label{eq_y},
\end{align}
where $x(t)\in\mathbb{R}^{n}$ denotes the state variable, $w(t)\in\mathbb{R}^{n}$ indicates the process noise, $y(t)\in\mathbb{R}^{p}$ represents the observation, and $v(t)\in\mathbb{R}^{p}$ denotes the observation noise. We assumed that $w(t)$ and $v(t)$ are independent white Gaussian noises with a mean of $0$ and identity covariance matrices, respectively. System matrices $A, G, C$, and $H$ are appropriately defined real matrices, and $HH^{\top}$ represents a positive-definite matrix. Throughout this study, we assume that $r \in \{ 1,\dots ,n-1\}$ exists such that $\mathrm{Re}(\lambda _{r}(A)) > \mathrm{Re}(\lambda _{r+1}(A))$. 
Then, the Kalman--Bucy filter for the system \eqref{eq_x}-\eqref{eq_y} is given as follows \cite{kalman1961new,bucy1987fsp}:
\begin{align}
\frac{d\hat{x}(t)}{dt}=&(A-\hat{P}(t){C}^{\top}{(HH^{\top})}^{-1}C)\hat{x}(t) \notag \\ 
& +\hat{P}(t){C}^{\top}{(HH^{\top})}^{-1}y(t)\label{eq_hatx},\\
\frac{d\hat{P}(t)}{dt}=&A\hat{P}(t)+\hat{P}(t){A}^{\top}+GG^{\top}\notag \\ 
&-\hat{P}(t){C}^{\top}{(HH^{\top})}^{-1}C\hat{P}(t) \label{eq_hatP}
\end{align}
with $\hat{x}(0)=\mathbb{E}[x(0)]$ and $\hat{P}(0) = \mathbb{E}[(x(0) - \hat{x}(0))(x(0)-\hat{x}(0))^{\top}]$, 
where $\hat{x}(t)\in\mathbb{R}^{n}$ denotes the estimate of $x(t)$ and its estimation error covariance matrix can be expressed as $\hat{P}(t)\in\mathbb{R}^{n\times n}$.  
The following proposition provides the steady-state solution to Equation \eqref{eq_hatP}.

\begin{prop}[{\cite[Cor. 5.4 and 5.5]{bucy1987fsp}}]\label{prop1}
If the system $(A, G, C)$ is controllable and observable, the following algebraic Riccati equation 
\begin{align*}
A\hat{P}_s+\hat{P}_sA^{\top}+GG^{\top}-\hat{P}_sC^{\top}{(HH^{\top})}^{-1}C\hat{P}_s = O_{n,n}
\end{align*}
yields a unique positive definite solution $\hat{P}_s$ and $A-\hat{P}_sC^{\top}{(HH^{\top})}^{-1}C$ is stable. Furthermore, the solution $\hat{P}(t)$ to Equation \eqref{eq_hatP} converges to $\hat{P}_s$.
\end{prop}

\subsection{Low-Rank approximated Kalman--Bucy filter}\label{subsec:LRKBF}
Implementing the Kalman--Bucy filter requires real-time numerical computation or offline calculation of the steady-state solution of the algebraic Riccati equation; both approaches are challenging in case of large system dimensions. To address this issue, we explore a modified LRKB filter \cite{yamada2021new}, hereafter referred to as the LRKB filter. Originating from \cite{bonnabel2012geometry}, this approach substitutes the positive (semi-)definite matrix $\hat{P}(t)$ with $\tilde{P}(t) = U(t) \tilde{R}(t) U(t)^{\top} {\in \mathbb{R}^{n\times n}}$, where $U(t) \in \mathrm{St}(r,n)$ and {$\tilde{R}(t) \in \mathbb{R}^{r\times r}$}, $\tilde{R}(t)>0$. $U(t)$ and $\tilde{R}(t)$ are the solutions of the following equations: 
\begin{align}
 \varepsilon\frac{dU(t)}{dt} =& (I_n-U(t){U}(t)^{\top})AU(t), \ U(0)^{\top}U(0)=I_r, \label{eq_U} 
 \\
  \frac{d\Tilde{R}(t)}{dt} =& A_{U(t)}\Tilde{R}(t) + \Tilde{R}(t) A_{U(t)}^{\top} + G_{U(t)}G_{U(t)}^{\top}\notag \\
 &-\tilde{R}(t) C_{U(t)}^{\top} {(HH^{\top})}^{-1}C_{U(t)}\Tilde{R}(t),\ \tilde{R}(0)>0, \label{eq_tildeR}
\end{align}
where $A_{U}\coloneqq U^{\top} A U{\in \mathbb{R}^{r\times r}}$, $C_{U} \coloneqq C U{\in \mathbb{R}^{p\times r}}$ and $G_{U} \coloneqq U^{\top}G {\in \mathbb{R}^{r\times n}}$. $\varepsilon\in(0,1]$ represents a small parameter to adjust the convergence speed of Equation \eqref{eq_U}. {Equation \eqref{eq_U} is expressed as the {\it Oja flow} \cite{OJA198569}, and $\varepsilon =1$ was set herein. 
Based on the straightforward calculation, $\frac{d}{dt}(U(t)^{\top}U(t)) =O_{r,r}$ if $U(t) \in \mathrm{St}(r,n)$ at $t\geq 0$; 
thus, $U(t) \in \mathrm{St}(r,n)$ whenever $U(0) \in \mathrm{St}(r,n)$.}
Equation \eqref{eq_tildeR} expresses an $r$-dimensional Riccati differential equation. 
The differential equation for the estimated value $\Tilde{x}(t) {\in \mathbb{R}^{n}}$ is as follows:
\begin{align}
\frac{d\Tilde{x}(t)}{dt}=&(A-\Tilde{P}(t){C}^{\top}{(HH^{\top})}^{-1}C)\Tilde{x}(t)\notag\\ 
&+{\Tilde{P}}(t){C}^{\top}{(HH^{\top})}^{-1}y(t)\label{eq_tildex}.
\end{align}
The differential equation for the error covariance matrix $\Tilde{V}(t)\coloneqq \mathbb{E}[(\Tilde{x}(t)-x(t)){(\Tilde{x}(t)-x(t))}^{\top}] {\in \mathbb{R}^{n\times n}}$ is as follows {\cite[Eq. (10)]{yamada2021comparison}}:
\begin{align}
 \frac{d\Tilde{V}(t)}{dt}=&(A-\Tilde{P}(t){C}^{\top}{(HH^{\top})}^{-1}C)\Tilde{V}(t)\notag \\
 &+\Tilde{V}(t){(A-\Tilde{P}(t){C}^{\top}{(HH^{\top})}^{-1}C)}^{\top}\notag \\
 &+GG^{\top}+\Tilde{P}(t){C}^{\top}{(HH^{\top})}^{-1}C\Tilde{P}(t) \label{eq_tildeV}.
\end{align}

Bonnabel et al. \cite{bonnabel2012geometry} employed $\lambda_{n}(GG^{\top})I_r$ instead of $G_{U(t)}G_{U(t)}^{\top} \in \mathbb{R}^{r\times r}$ in Equation \eqref{eq_tildeR}. However, when $GG^{\top}$ is singular, indicating $\lambda _{n}(GG^{\top})=0$, the estimation error of Bonnabel’s LRKB filter may increase if $(A,G)$ is controllable. To address this limitation, Yamada et al. \cite{yamada2021new} refined Bonnabel’s filter and provided conditions for bounded estimation error for LTI systems with symmetric $A$. 

The analysis was initially restrictive, { and a numerical simulation in \cite{yamada2021comparison} demonstrated that the proposed filter could ensure bounded estimation errors for a general square matrix $A$; thus, this study extends the analysis of the modified LRKB filter to general cases.} 

\section{Equilibrium points and convergence of the Oja flow}\label{subsec:Oja equi}
First, we elaborate on all the equilibrium points of the Oja flow \eqref{eq_U}. In addition to their local stability, this section describes the domains of attraction for the stable equilibrium points.

{Throughout this section, we describe the following notations: $\Psi := [\bm{\psi}_1(A), \dots, \bm{\psi}_n(A)] {\in \mathbb{C}^{n\times n}}$ and $\Lambda :={\Psi}^{-1}A\Psi \in \mathbb{C}^{n\times n}$.  $\Lambda $ denotes the Jordan form of $A$.}

\subsection{Equilibrium points of the Oja flow}\label{subsubsec:Oja equi}
We present several lemmas to examine the structure of the equilibrium points of the Oja flow \eqref{eq_U}. 
The following lemma holds true:
\begin{lem}\label{lem:K(t)_str}
Let $K \in \mathbb{C}^{n \times r}$ be $K=\Psi ^{-1}U$, where $U \in \mathrm{St}(r,n)$. Thereafter, $K$ presumes the form 
\begin{align}
K = P \begin{bmatrix} K_r \\ O_{n-r,r} \end{bmatrix}, \label{eq_K_str}
\end{align}
where $P \in \mathbb{C}^{n \times n}$ represents a unitary matrix, and $K_r \in \mathbb{C}^{r \times r}$ denotes a regular matrix. 
\end{lem}

\begin{proof}
 As $r$ denotes the rank of $U\in \mathrm{St}(r,n)$, the rank of $K$ is $\mathrm{rank}(K) = \mathrm{rank}({\Psi}^{-1}U) = r$. According to singular value decomposition, this statement holds. 
\end{proof}

For a unitary matrix $P\in \mathbb{C}^{n\times n}$, let $\Sigma ^{P}:=\begin{bsmallmatrix} \Sigma_{11}^{P} & \Sigma_{12}^{P} \\ \Sigma_{21}^{P} & \Sigma_{22}^{P} \end{bsmallmatrix} := {P}^{\dagger}{\Psi}^{\dagger}\Psi P $, where $\Sigma_{11}^{P} \in \mathbb{C}^{r \times r}$, $\Sigma_{12}^{P}=(\Sigma _{21}^{P})^{\dagger} \in \mathbb{C}^{r\times (n-r)}$, and $\Sigma_{22}^{P} \in \mathbb{C}^{(n-r) \times (n-r)}$. If $P=I_{n}$, the superscript is omitted.  
Then, the following lemma holds for $K_r$ {of \eqref{eq_K_str}}.

\begin{lem}\label{lem:K_r(t)_str}
Let $K_r {\in \mathbb{C}^{r\times r}}$ be of Lemma \ref{lem:K(t)_str}. Then, $K_{r}$ is represented by {$\Sigma_{11}^{P}$ and} a unitary matrix $W_{{P}}\in \mathbb{C}^{r\times r}$ such that $K_r = (\Sigma_{11}^{P})^{-1/2}W_{{P}}$. 
For any orthogonal matrix $W \in \mathbb{R}^{r\times r}$, the following equation also satisfies \eqref{eq_K_str}: $K_r = (\Sigma_{11}^{P})^{-1/2}W_{{P}} W$. 
\end{lem}

\begin{proof}
From polar decomposition, a unitary matrix $W_{{P}}$ exists such that $K_{r} = (K_{r}K_{r}^{\dagger})^{1/2} W_{{P}}$.  
Because $U^{\top}U = U^{\dagger}U = I_r$, then $\begin{bsmallmatrix} K_{r}^{\dagger} && O_{r,n-r} \end{bsmallmatrix} {P}^{\dagger}{\Psi}^{\dagger}\Psi P \begin{bsmallmatrix} {K_r} \\ O_{n-r,r} \end{bsmallmatrix} = I_r$. Therefore, $K_{r}^{\dagger} \Sigma_{11}^{P} K_r = I_r$ and $K_{r}=(\Sigma_{11}^{P})^{-1/2}W_{{P}}$. 
Because $U\in \mathrm{St}(r,n) \subset \mathbb{R}^{n\times r}$, $K_{r}=(\Sigma_{11}^{P})^{-1/2}W_{{P}}W$ for any orthogonal matrix $W \in \mathbb{R}^{r\times r}$ always satisfies $U^{\top}U=I_{r}$.  
Therefore, the statement holds.  
\end{proof}

Using Lemmas \ref{lem:K(t)_str} and \ref{lem:K_r(t)_str}, the following proposition holds for the equilibrium points of the Oja flow \eqref{eq_U} with a real square matrix $A$. 
\begin{prop}\label{prop:equili_str}
The equilibrium point is expressed as $\bar{U}_{P} = \Psi P \begin{bsmallmatrix} K_r \\ O_{n-r,r} \end{bsmallmatrix}\in\mathrm{St}(r, n)$, where $P \in \mathbb{R}^{n \times n}$ denotes a permutation matrix and $K_r \in \mathbb{C}^{r \times r}$ represents the form in Lemma \ref{lem:K_r(t)_str}. 
Furthermore, any $\bar{U}_{P}^{\prime}$ of $\mathcal{U}_{P} := \{ \bar{U}_{P}W \ | \ W \in \mathbb{R}^{r\times r}, \ W^{\top}W=I_{r} \}$ is an equilibrium point. 
\end{prop}

We express $\bar{U}$ and $\mathcal{U}$ if $P=I_{n}$. { Note that there exists a unitary $P$ such that $\mathcal{U}_{P} = \mathcal{U}$; refer to the following proof. }

\begin{proof}
From Lemma \ref{lem:K(t)_str}, {for any $\bar{U}_{P} \in \mathrm{St}(r,n)$ there exist unitary $P\in\mathbb{C}^{n\times n}$ and regular $K_{r}\in \mathbb{C}^{r\times r}$ such that} $\bar{U}_{P} = \Psi P \begin{bsmallmatrix} K_r \\ O_{n-r,r} \end{bsmallmatrix}$. Let $\Lambda_P := P^{\dagger}\Lambda P {\in \mathbb{C}^{n\times n}} $. 
{To satisfy $\bar{U}_{P}$ to be an equilibrium of Eq. \eqref{eq_U}, }
\begin{align*}
0=& (I_n-\bar{U}_{P}{\bar{U}_{P}}^{\top})A\bar{U}_{P} \\
=&\Psi P \left( \Lambda_P - \begin{bmatrix} {({\Sigma_{11}^{P}})^{-1}} & O_{r,n-r} \\ O_{n-r,r} &O_{n-r,n-r} \end{bmatrix} 
\Sigma ^{P} \Lambda_P \right) 
\begin{bmatrix} K_r \\ O_{n-r,r} \end{bmatrix}
\end{align*}
{needs to hold.  Here, we use Lemma \ref{lem:K_r(t)_str} in the second equality.} 
As $\Psi P$ is a regular matrix, the following equation holds:
\begin{align*}
0&=\left( \Lambda_P - \begin{bmatrix} I_r & ({\Sigma_{11}^{P}})^{-1}{\Sigma_{12}^{P}}\\ O_{n-r,r} & O_{n-r,n-r} \end{bmatrix}\Lambda_P \right) 
\begin{bmatrix} K_r \\ O_{n-r,r} \end{bmatrix}.
\end{align*}
The equality holds if $\Lambda_P$ denotes a block diagonal matrix. 
To satisfy the block diagonality of $\Lambda_P$, a unitary matrix $P$ should be the product of permutation matrix $P_{\rm per}$ and block diagonal unitary matrix $P_{\rm BD} = \mathrm{block}\mbox{-}\mathrm{diag}[P_{r}, P_{\perp}]${; $P=P_{\rm per}P_{\rm BD}$}, where $P_{r}\in \mathbb{C}^{r\times r}$ and $P_{\perp}\in \mathbb{C}^{(n-r)\times (n-r)}$ are unitary matrices. 
{ As $(\Sigma _{11}^{P})^{-1/2} = P_{r}^{\dagger} (\Sigma _{11}^{P_{\rm per}})^{-1/2} P_{r}$, $P\begin{bsmallmatrix} (\Sigma _{11}^{P})^{-1/2} W_{P} \\ O_{n-r,r} \end{bsmallmatrix} = P_{\rm per} \begin{bsmallmatrix} (\Sigma _{11}^{P_{\rm per} })^{-1/2} P_{r}W_{P} \\ O_{n-r,r} \end{bsmallmatrix}$ holds, which enables us to} rewrite $P_{r}K_{r}$ with $K_{r}$; thus, matrix $P$ is a permutation matrix without loss of generality. 
From Lemma \ref{lem:K_r(t)_str}, if $\bar{U}_{P}$ is also an equilibrium point, then for any orthogonal matrix $W\in \mathbb{R}^{r\times r}$, $\bar{U}_{P}W$ is also an equilibrium point.  
The proof is completed.
\end{proof}

\begin{rem}
In \cite{bonnabel2012geometry}, a stable equilibrium set for a general $A$ matrix was characterized by $\bm{\psi}_{i}(A+A^{\top})$, $i=1,\dots ,r$.  
Proposition \ref{prop:equili_str} provides an alternative and direct characterization of the equilibrium sets of the Oja flow. 
Note that Proposition \ref{prop:equili_str} does not deny the existence of other invariant sets such as limit cycles.  
\end{rem}

Any $\bar{U}_{P} \in \mathcal{U}_{P}$ retains certain eigenvalues of $A$.  

\begin{prop}\label{prop:UTAUeigen}
Let $P{\in\mathbb{R}^{n\times n}}$ be any permutation matrix and $\bar{U}_{P} \in \mathcal{U}_{P}$. 
{Then, $ \{ \lambda_i({\bar{U}_{P}}^{\top}A\bar{U}_{P}) \} _{i=1}^{r} = \{ \lambda_{\mathcal{I}_{P}(i)}(A)\} _{i=1}^{r} $, where $\mathcal{I}_P$ is a permutation related to $P$.}
\end{prop}
\begin{proof}
Let ${{\Psi}_{P,r}} := \Psi P\begin{bsmallmatrix} I_r \\ O_{n-r,r} \end{bsmallmatrix} \in \mathbb{C}^{n \times r}$, ${{\Lambda}_{P,r}} := \begin{bsmallmatrix} I_r && O_{r,n-r} \end{bsmallmatrix} P^{\top}\Lambda P\begin{bsmallmatrix} I_r \\ O_{n-r,r} \end{bsmallmatrix}\in \mathbb{C}^{r \times r}$. 
Thereafter, $\bar{U}_{P}^{\top}\bar{U}_{P} = K_{r}^{\dagger}{{\Psi}_{P,r}^{\dagger}} {{\Psi}_{P,r}}K_r=I_r$. {As $K_r \in \mathbb{C}^{r\times r}$ is regular, $K_{r}^{-1}$ exists. Therefore, we can obtain the following equation:}
{\begin{align*}
\bar{U}_{P}^{\top} A \bar{U}_{P} 
= K_{r}^{\dagger}{{\Psi}_{P,r}^{\dagger}}{{\Psi}_{P,r}} K_r K_{r}^{-1}{{\Lambda}_{P,r}}  K_r
=K_{r}^{-1} {{\Lambda}_{P,r}}  K_r. 
\end{align*}}
As the similarity transformation preserves the eigenvalues, {$\{ \lambda_i({\bar{U}_{P}}^{\top} A \bar{U}_{P}) \}_{i=1}^{r} = \{ \lambda_{\mathcal{I}_{P}(i)}(A) \} _{i=1}^{r}$} holds.
\end{proof}

The preservation of the part of the eigenvalues of $A$ is crucial for analyzing the controllability and observability of the reduced system matrices in Props. \ref{prop:lrde_observe} and \ref{prop:lrde_reach}.

\subsection{Convergence and domain of attraction of the Oja flow}\label{subsubsec:Oja conver}
In this section, we analyze the local stability of each $\mathcal{U}_{P}$ defined in Prop. \ref{prop:equili_str} and the domain of attraction of the Oja flow. 
Accordingly, we set $K(t)=\begin{bsmallmatrix}
K_r(t) \\ K_{\perp}(t)    
\end{bsmallmatrix} = \Psi ^{-1} U(t),$
where $K_r(t) \in \mathbb{C}^{r \times r}, K_{\perp}(t)\in\mathbb{C}^{(n-r) \times r}$.  

\begin{thm}\label{thm:equili_localconv}
The set $\mathcal{U}$ is locally asymptotically stable. Furthermore, any equilibrium set $\mathcal{U}_{P} \neq \mathcal{U}$ is unstable.
\end{thm}

\begin{proof}
We consider the perturbations from the equilibrium point $\Psi P\begin{bsmallmatrix} K_r \\ O_{n-r,r} \end{bsmallmatrix}$ given by $\Psi P\begin{bsmallmatrix} \delta K_r (t) \\ \delta K_{\perp}(t) \end{bsmallmatrix}$. 
We define $\mathrm{block}\mbox{-}\mathrm{diag} [{\Lambda_{P,r}, \Lambda_{P,\perp}}] := P^{\top}\Lambda P$. 
{ As $\lim _{t\to \infty}\delta K_{\perp} (t) = O_{n-r,r}$ is sufficient for this statement, we focus only on $\frac{d}{dt}\delta K_{\perp} (t)$. }
Neglecting higher-order terms, the differential equation for $\delta K_{\perp}(t)$ is stated as follows:
\begin{align*}
\frac{d\delta K_{\perp}(t)}{dt} 
=& {\Lambda_{P,\perp}}\delta K_{\perp}(t) - \delta K_{\perp}(t)K_r^{-1}{\Lambda_{P,r}} K_r.    
\end{align*}
The vectorization of the matrices yields 
\begin{align*}
& \frac{d\mathrm{vec}(\delta K_{\perp}(t))}{dt} \\ 
=& (I_r\otimes {\Lambda_{P, \perp}} -(K_r^{\top} {\Lambda_{P,r}^{\top}} K_{r}^{-\top})\otimes I_{n-r})\mathrm{vec}(\delta K_{\perp}(t)),
\end{align*}
where $\otimes$ denotes the Kronecker product. 
{If $P=I_{n}$, $\mathrm{Re}(\lambda _{1}(I_r\otimes\Lambda_{\perp }-(K_r^{\top} \Lambda_{r}^{\top} K_{r}^{-\top})\otimes I_{n-r})) = \mathrm{Re}(\lambda_{r+1}(A)) - \mathrm{Re}(\lambda_{r}(A)) <0 $; therefore, }
$\delta K_{\perp}(t)$ converges to $O_{n-r,r}$, and consequently, $K_{r} + \delta K_{r}(t)$ converges to $K_{r}W$, where $W \in \mathbb{R}^{r\times r}$ denotes an orthogonal matrix. 
Therefore, $\mathcal{U}$ is asymptotically stable. 
In contrast, any element in $\mathcal{U}_{P} \neq \mathcal{U}$ leaves $\mathcal{U}_{P}$ {when it is perturbed}. Therefore, any set $\mathcal{U}_{P} \neq \mathcal{U}$ is unstable. 
\end{proof}

Theorem \ref{thm:equili_localconv} suggests that if there is no other stable invariant set such as limit cycle, any solution $U(t)$ of Eq. \eqref{eq_U} from a nonequilibrium point converges to $\mathcal{U}$; therefore, investigating $\mathcal{U}$ is crucial to understanding the properties of the LRKF. 

As Theorem \ref{thm:equili_localconv} only ensures local stability, establishing the domain of the attraction of $\mathcal{U}$ is expected. 
If we consider $(n,r)=(2,1)$, the global behavior of the Oja flow is stated as follows. 

\begin{prop}\label{prop:equili_globalconvn=2}
Suppose $(n,r)=(2,1)$. If $K(0) \neq \begin{bsmallmatrix} 0 ,& \pm \Sigma _{22}^{-1/2} \end{bsmallmatrix}^{\top}$, then the solution ${U(t)=}\Psi K(t)$ of {\eqref{eq_U}} converges to $\mathcal{U}=\{ \pm \Sigma _{11}^{-1/2} \psi _{1}(A)  \}$.
\end{prop}

\begin{proof}
For $A\in \mathbb{R}^{2\times 2}$, $\mathrm{Re}(\lambda _{1}(A)) > \mathrm{Re}(\lambda _{2}(A))$ implies $\lambda _{i}$ are real {and $\Psi \in \mathbb{R}^{2\times 2}$}. Therefore, $\Sigma = {\Psi ^{\top}\Psi } \in \mathbb{R}^{2\times 2}$. 
Here, $K(t)$ moves along an ellipse $\{ K \in \mathbb{R}^{2} \ | \ K^{\top} \Sigma K=1 \}$. 
Based on $K(t)^{\top}\Sigma K(t) = \Sigma _{11}K_{r}(t)^{2} + \Sigma _{22}K_{\perp}(t)^{2} + 2 \Sigma _{12} K_{r}(t) K_{\perp}(t)=1$, 
\begin{align}
\frac{d}{dt}K(t) 
=& {(I_{2} - K(t)K(t)^{\top} \Sigma ) \Lambda K(t) } \nonumber
\\
= & \frac{\delta \lambda}{2}
\left\{
\begin{bmatrix}
0
\\
-2 K_{\perp}(t)
\end{bmatrix}
+
\gamma (K_{r}(t),K_{\perp}(t))
\begin{bmatrix}
K_{r}(t)
\\
K_{\perp}(t)
\end{bmatrix}
\right\} , \label{eq:nr21_DE}
\end{align}
where $\delta \lambda := \lambda _{1}(A) - \lambda _{2}(A) $ and $\gamma (K_{r},K_{\perp}):=(1 + \Sigma _{22}K_{\perp}^{2} - \Sigma _{11}K_{r}^{2})$. Equation \eqref{eq:nr21_DE} contains four equilibrium points: $(\bar{K}_{r},\bar{K}_{\perp}) = (\pm \Sigma ^{-1/2}_{11},0)$ and $(0,\pm \Sigma ^{-1/2}_{22})$. Herein, we consider the following four cases: 
(i) If $\gamma (K_{r}(t),K_{\perp}(t)) >2$, then $|K_r(t)|$ and $|K_{\perp}(t)|$ increase.  
(ii) If $\gamma (K_{r}(t),K_{\perp}(t)) \in (0,2]$, then $|K_r(t)|$ increases and $|K_{\perp}(t)|$ decreases. 
(iii) If $\gamma (K_{r}(t),K_{\perp}(t)) =0$, then $K_{r}(t)$ does not change and $|K_{\perp}(t)|$ decreases or does not change if $K_{\perp}(t)=0$. 
(iv) If $\gamma (K_{r}(t),K_{\perp}(t)) <0$, then $|K_r(t)|$ and $|K_{\perp}(t)|$ decrease.  
In (i), the rate of increase of $|K_{r}(t)|$ is larger than that of $|K_{\perp}(t)|$. Then, Equation \eqref{eq:nr21_DE} enters case (ii) and becomes case (iii). After case (iii), Equation \eqref{eq:nr21_DE} does not return to case (ii); thus, $| K_{\perp}(t)|$ converges to $0$.  
Based on the constraint $K(t)^{\top}\Sigma K(t) =1$, this phenomenon implies that $K_{r}(t)$ converges to $\pm \Sigma _{11}^{-1/2}$. 
\end{proof}

Unlike Proposition \ref{prop:equili_globalconvn=2}, estimating the exact domain of attraction for a general $(n,r)$ is challenging. We provide sufficient conditions for the domain of attraction.  
Before proceeding, we introduce the following lemmas: 

\begin{lem} \label{lem:matrix_inequality}
For any matrix $X, Y \in \mathbb{C}^{m \times n}$ and any real number $\alpha > 0$, the following inequalities hold:
\begin{align*}
\pm (X^{\dagger}Y+Y^{\dagger}X) \leq \alpha X^{\dagger}X+\frac{1}{\alpha}Y^{\dagger}Y. 
\end{align*}   
\end{lem}

\begin{proof}
For any $X, Y\in \mathbb{C}^{m\times n}$, we obtain 
$\begin{bmatrix}
X & Y
\end{bmatrix}^{\dagger}
\begin{bmatrix}
X & Y
\end{bmatrix} \geq 0$.
By multiplying both sides of the inequality by $[\sqrt{\alpha} I_n, \ \pm \frac{1}{\sqrt{\alpha}}I_n ]$ and its transposition from the left and right sides, we obtain Lemma \ref{lem:matrix_inequality}. 
\end{proof}

\begin{lem}\label{lem:Schur_str}
For any $A \in \mathbb{R}^{n \times n}$, a unitary matrix $S \in \mathbb{C}^{n\times n}$ and upper triangle matrix $Q\in \mathbb{C}^{n\times n}$ exist such that $S^{\dagger}AS$ becomes an upper triangle and satisfies $S = \Psi Q$.  
\end{lem}

\begin{proof}
Based on the QR decomposition of $\Psi$, $\Psi = S T$, where $T$ denotes the upper triangle and regular matrix. As $Q=T^{-1}$ is also an upper triangle, this statement holds. 
\end{proof}

If the matrix $A$ is decomposed using the unitary matrix $S$ satisfying Lemma \ref{lem:Schur_str}, we can consider $S^{\dagger}AS = \begin{bsmallmatrix} L_{11} && L_{12} \\ O_{n-r,r} && L_{22} \end{bsmallmatrix}$, where $L_{11} {\in \mathbb{C}^{r\times r}}$ denotes an upper triangular matrix with eigenvalues $\lambda_1(A), \dots, \lambda_r(A)$, $L_{22}{\in \mathbb{C}^{(n-r)\times (n-r)}}$ represents an upper triangular matrix with eigenvalues $\lambda_{r+1}(A), \dots, \lambda_n(A)$, and $L_{12} {\in \mathbb{C}^{r \times (n-r)}}$. 
Note that $\mathrm{St}(r, n) = \{S\begin{bsmallmatrix} F_1 \\ F_2 \end{bsmallmatrix} \in \mathbb{R}^{n\times r} \mid F_1 \in \mathbb{C}^{r \times r}, F_2 \in \mathbb{C}^{(n-r)\times r},\ F_{1}^{\dagger}F_1 +F_{2}^{\dagger}F_2 = I_r\}$. 

\begin{thm}\label{thm:equili_globalconv}
If ${\lambda}^{l_1}_r \coloneqq \lambda_r(L_{11} + L_{11}^{\dagger}) > {\lambda}^{l_2}_1 \coloneqq \lambda_{1}(L_{22} + L_{22}^{\dagger})$, and $U(0) \in \mathcal{V}_{\beta}:= \{S\begin{bsmallmatrix} F_1 \\ F_2 \end{bsmallmatrix} \in \mathrm{St}(r,n) \mid F_1 \in \mathbb{C}^{r \times r}, F_2 \in \mathbb{C}^{(n-r)\times r}, \ 0 \le F_{2}^{\dagger}F_2 < \beta I_r\}$, then the solution to Equation \eqref{eq_U} converges to an element of $\mathcal{U}$. Here, $\beta = ( 1+ 4 \ell _{\max} / ({\lambda}^{l_1}_r - {\lambda}^{l_2}_1)^2)^{-1} $ and $\ell_{\max}$ is the maximum eigenvalue of $L_{12}^{\dagger}L_{12}$.
\end{thm}

\begin{proof}
By using Lemma \ref{lem:Schur_str}, $U(t)$ can be expressed as $U(t) = \Psi Q 
\begin{bmatrix} F_{1}(t)^{\top} & F_{2}(t)^{\top} \end{bmatrix}^{\top} $. 
We demonstrate the initial domain is such that $\lim _{t\to \infty}F_{2}(t)=O_{n-r,r}$, implying $U(t)$ converges to $\mathcal{U}$.  
The differential equation of $\mathrm{Tr}[{F_{2}(t)}^{\dagger}F_{2}(t)]$ is
\begin{align*}
& \frac{d}{dt}\mathrm{Tr}[{F_{2}(t)}^{\dagger}F_{2}(t)]
\\
=& \mathrm{Tr}[{F_{2}(t)}^{\dagger}(L_{22} + L_{22}^{\dagger})F_{2}(t) \{ I_{r} - {F_{2}(t)}^{\dagger}F_{2}(t)\} ] \\
&-\mathrm{Tr}[{F_{2}(t)}^{\dagger}F_{2}(t) F_{1}(t)^{\dagger}(L_{11} + L_{11}^{\dagger})F_1(t)] \\
&-\mathrm{Tr}[{F_{2}(t)}^{\dagger}F_{2}(t) \{ {F_{1}(t)}^{\dagger}L_{12}{F_{2}(t)} + {F_{2}(t)}^{\dagger} L_{12}^{\dagger}{F_1(t)} \} ] .\\
\end{align*}
By applying Lemma \ref{lem:matrix_inequality} with $\alpha \in (0,\delta \lambda )$, where $\delta \lambda := {\lambda}^{l_1}_r-{\lambda}^{l_2}_1$, and ${F_1(t)}^{\dagger}{F_1(t)}= I_{r}-{F_{2}(t)}^{\dagger}F_{2}(t)$ to the aforementioned equation, the following inequality holds:
\begin{align*}
& \frac{d}{dt}\mathrm{Tr}[{F_{2}(t)}^{\dagger}F_{2}(t)] \notag \\
\leq & -\delta \lambda \mathrm{Tr}[{F_{2}(t)}^{\dagger}F_{2}(t) \{ I_{r} - {F_{2}(t)}^{\dagger}F_{2}(t)\} ]\notag \\
&+ \frac{1}{\alpha}\mathrm{Tr}[{F_{2}(t)}^{\dagger}F_{2}(t){F_{2}(t)}^{\dagger} L_{12}^{\dagger} L_{12} {F_{2}(t)}] \notag \\
&+ \alpha\mathrm{Tr}[{F_{2}(t)}^{\dagger}F_{2}(t) \{ I_{r} - {F_{2}(t)}^{\dagger}F_{2}(t)\} ] \notag 
\\ \leq &-(\delta \lambda-\alpha)\mathrm{Tr}[{F_{2}(t)}^{\dagger}F_{2}(t)] \notag\\
&+\left(\delta \lambda -\alpha +\frac{\ell_{\max}}{\alpha} \right) \mathrm{Tr}[ ( {F_{2}(t)}^{\dagger} F_{2}(t) )^2 ].
\end{align*}
We consider the following Riccati differential equation: 
\begin{align}
\frac{d}{dt} Z(t) 
= & -(\delta \lambda-\alpha)Z(t)
+\left( \delta \lambda -\alpha + \frac{\ell_{\max}}{\alpha} \right)
Z(t)^2 , \label{eq:Z_2}
\end{align}
where $Z(t)=Z(t)^{\dagger}\in \mathbb{C}^{r \times r}$ and $Z(0) = F_{2}(0)^{\dagger}F_{2}(0)$. The solution of Equation \eqref{eq:Z_2} converges to $O_{r,r}$ if $Z(0) < \frac{\alpha (\delta \lambda -\alpha)}{\alpha (\delta \lambda -\alpha) + \ell_{\max}} I_{r}$.  
Since $\mathrm{Tr}[F_{2}(t)^{\dagger} F_{2}(t)] \leq \mathrm{Tr}[Z(t)]$ for $t\geq 0$, if $F_{2}(0)^{\dagger} F_{2}(0) < \frac{\alpha (\delta \lambda -\alpha)}{\alpha (\delta \lambda -\alpha) + \ell_{\max}} I_{r}$, then $F_{2}(t)$ converges to $O_{n-r,r}$, implying that $U(t)$ converges to $\mathcal{U}$. 

Subsequently, we estimate the maximum domain of attraction. Let 
\begin{align*}
f(\alpha) := \frac{\alpha (\delta \lambda -\alpha)}{\alpha (\delta \lambda -\alpha) + \ell_{\max}}
, \ \alpha \in (0,\delta \lambda )
.
\end{align*}
Thereafter, $f$ is a concave, positive function over $(0,\delta \lambda )$ and is 
maximum when $\alpha _{\max} = \delta \lambda /2$.  
Thus, $f(\alpha _{\max}) = \frac{\delta \lambda ^2}{\delta \lambda ^2 + 4 \ell _{\max}} = \beta$, thereby completes the proof. 
\end{proof}

By Bendixson’s inequality, ${\lambda}^{l_1}_r > {\lambda}^{l_2}_1$ implies $\mathrm{Re}(\lambda_r(A)) > \mathrm{Re}(\lambda_{r+1}(A))$. 
A prior numerical study \cite{yamada2021comparison} suggests that solution $U(t)$ to Eq. \eqref{eq_U} converges to $\mathcal{U}$ for almost initial matrices if $\mathrm{Re}(\lambda_r(A)) > \mathrm{Re}(\lambda_{r+1}(A))$ is satisfied, implying that the conditions in Theorem \ref{thm:equili_globalconv} can be relaxed.
Theorem \ref{thm:equili_globalconv} is similar to Proposition \ref{prop:equili_globalconvn=2} if $\ell_{\max} =0$, satisfied when $A$ is a normal matrix:

\begin{cor}\label{cor:equili_globalconv_normal}
Let $A$ denote a normal matrix. If $U(0) \in \mathcal{V}_{1}$, then the solution to Eq. \eqref{eq_U} converges to an element of $\mathcal{U}$.
 
\end{cor}
\begin{proof}
For a normal matrix $A\in \mathbb{R}^{n\times n}$, $S$ diagonalizes $A$. Therefore, $\ell _{max}=0$. This statement follows from Thm. \ref{thm:equili_globalconv}. 

\end{proof}

{The volume of $\mathrm{St}(r,n)$ can be precisely calculated \cite[Sec. 1.4.4]{Chikuse2003}. The same calculation reveals that $\mathcal{V}_{1}$ has the same volume as $\mathrm{St}(r,n)$, implying that almost all elements of $\mathrm{St}(r,n)$ are in $\mathcal{V}_{1}$. Thus, if $U(0)$ is generated from, e.g., the uniform distribution on $\mathrm{St}(r,n)$ \cite[Thm. 2.2.1]{Chikuse2003}, then $U(t)$ converges to $\mathcal{U}$ for almost all samples. 

Demonstrating the equality of two volumes exceeds the scope of this paper, but an illustrative example is provided instead. For $A = A^{\top} \in \mathbb{R}^{3\times 3}$ and $(n,r) = (3,1)$, $\mathcal{U}$ consists of the North and South poles of a unit sphere by choosing a suitable coordinate system.
$\mathcal{V}_{1}$ is then equivalent to the surface area of the sphere excluding the equator line; hence, $\mathcal{V}_{1}$ possesses the same volume as $\mathrm{St}(1,3)$.
However, for a general $A \in \mathbb{R}^{3\times 3}$, the unstable equilibrium points do not reside on the equator line, even if the North and South poles are stable equilibrium, presenting a challenge in precisely estimating the domain of attraction for $\mathcal{U}$. 
}

%%%%%%%%%%%%%%%%%%%%%%%%%%%%%%%%%%%%%%%%%%%%%%%%%%%%%%%%%%%%%%%%%%%%%%%
\section{Boundedness of estimation error}\label{subsec:bounded LKBF}
To ensure a bounded estimation error for the LRKB filter, we define the conditions for the steady-state solution of Eq. \eqref{eq_tildeV} as in \cite{yamada2021new} that addressed symmetric $A$.
{As the Oja flow can be computed offline for LTI systems, we assume $U(0) = \bar{U} \in \mathcal{U}$ throughout this section. }
Initially, we started the analysis by examining the controllability and observability using the following equality verified in Prop. \ref{prop:equili_str}:
\begin{align}
{\bar{U} \bar{U}^{\top} A \bar{U} = A \bar{U}\quad \forall \bar{U} \in \mathcal{U}.} \label{eq:UU_top_AU=AU}
\end{align}

\begin{prop}\label{prop:lrde_observe}
If $(C,A) { \in \mathbb{R}^{p\times n} \times \mathbb{R}^{n\times n}}$ is observable, then $(C_{\bar{U}},A_{\bar{U}}){ \in \mathbb{R}^{p\times r} \times \mathbb{R}^{r\times r}}$ is observable as well.
\end{prop}

\begin{proof}
If $(C_{\bar{U}},A_{\bar{U}})$ is unobservable, then a nonzero vector $\bm{v}\in{\mathbb{C}^r}$ and $\lambda _{i}$ exist such that $\bar{U}^{\top}A\bar{U}\bm{v} =\lambda_{i}\bm{v}$, $C\bar{U}\bm{v}=0$. 
Based on Prop. \ref{prop:UTAUeigen}, $\lambda _{k} = \lambda_{k}(A)$, $k=1,\dots ,r$.  
From Eq. \eqref{eq:UU_top_AU=AU}, we can state 
\begin{align*}
\bar{U}\bar{U}^{\top} A \bar{U} \bm{v}=A\bar{U}\bm{v}= \lambda _{i} \bar{U} \bm{v},\quad C \bar{U} \bm{v}=0 .
\end{align*}
Since the column vectors of $\bar{U}$ are linearly independent, $\bar{U}\bm{v} \neq 0$. This result contradicts the observability of $(C,A)$. Therefore, the statement holds. 
\end{proof}

\begin{prop}\label{prop:lrde_reach}
If $(A,G){ \in \mathbb{R}^{n\times n} \times \mathbb{R}^{n\times n}}$ is controllable, then $(A_{\bar{U}},G_{\bar{U}}){ \in \mathbb{R}^{r\times r} \times \mathbb{R}^{r\times n}}$ is also controllable.
\end{prop}

\begin{proof}
If $(A,G) $ is controllable, we can obtain $
\mathrm{rank}
\begin{bmatrix}
G & A-\lambda_i I_n
\end{bmatrix}
=n
$ for any $\lambda_i = \lambda_i(A)$, $i=1,\dots,n$. 
Subsequently, we consider the rank of the following matrix:
\begin{align*}
\bar{U}^{\top}
\begin{bmatrix}
G & A - \lambda_{i} I_n
\end{bmatrix}
\begin{bmatrix}
I_n & O_{n,r}\\
O_{n,n} & \bar{U}
\end{bmatrix}
.
\end{align*}
As $\bar{U}$ is full rank, the following equation holds:
\begin{align*}
\mathrm{rank}
\begin{bmatrix}
\bar{U}^{\top}G & \bar{U}^{\top} A \bar{U} - \lambda _{i} I_r
\end{bmatrix}
= r \quad \forall \lambda_{i},\ i=1,\dots ,n .
\end{align*}
This signifies that $(A_{\bar{U}},G_{\bar{U}})$ is controllable. Therefore, if $(A,G)$ is controllable, then $(A_{\bar{U}},G_{\bar{U}})$ is also controllable.
\end{proof}

\begin{rem}
If $(A, G, C)$ is stabilizable and detectable, then $(A_{\bar{U}}, G_{\bar{U}}, C_{\bar{U}})$ are also stabilizable and detectable. The proof adheres to the procedures outlined in Propositions \ref{prop:lrde_observe} and \ref{prop:lrde_reach}.
\end{rem}

If Propositions \ref{prop:lrde_observe} and \ref{prop:lrde_reach} are applicable, then Prop. \ref{prop1} can be applied to Eq. \eqref{eq_tildeR}. Consequently, $\tilde{R}(t) {\in \mathbb{R}^{r\times r}}$ converges to $\tilde{R}_{\bar{U}} > 0$ and $A_{\bar{U}} - \tilde{R}_{\bar{U}} C_{\bar{U}}^{\top} {(HH^{\top})}^{-1}C_{\bar{U}} {\in \mathbb{R}^{r\times r}}$ is stable.
However, $\tilde{R}(t)$ can converge to different solutions depending on the choice of $\bar{U}$. Even in such cases, the following proposition ensures the uniqueness of $\tilde{P}_{r}:= \lim _{t \to \infty} \tilde{P}(t) {\in \mathbb{R}^{n\times n}}$.

\begin{prop}\label{prop:lrde_unique}
Suppose $(A, G, C)$ is controllable and observable. 
Let $U_1, U_2 \in \mathcal{U}$. Subsequently, $U_1 \tilde{R}_{U_1} U_{1}^{\top}=U_2\tilde{R}_{U_2} U_{2}^{\top} = \tilde{P}_{r}$ holds. 
\end{prop}

The proof follows the proof of Prop. 4 in \cite{yamada2021new}.

Subsequently, we present a proposition concerning the stability of the estimation error system and provide the conditions for the boundedness of the estimation error.

\begin{prop}\label{prop:lrkf_bound}
{Let $\sigma_{i} = \lambda _{i}(A_{\bar{U}}-\tilde{R}_{\bar{U}}C_{\bar{U}}^{\top}{(HH^{\top})}^{-1}C_{\bar{U}})$ for $i=1,\dots,r$.  
}
Then the eigenvalues of $A-\tilde{P}_{r}{C}^{\top}{(HH^{\top})}^{-1}C {\in\mathbb{R}^{n\times n}}$ are $\sigma_1,\dots,\sigma_r$, $\lambda_{r+1}(A),\dots, \lambda_n(A)$.
\end{prop}

\begin{proof}
{From Lemma \ref{lem:Schur_str}, there exists a unitary $S \in \mathbb{C}^{n\times n}$ such that $S^{\dagger}AS = \begin{bsmallmatrix} L_{11} & L_{12} \\ O_{n-r,r} & L_{22}\end{bsmallmatrix}$, where $L_{11} \in \mathbb{C}^{r\times r}$ with the eigenvalues $\lambda _{1}(A),\dots , \lambda _{r}(A)$, $L_{22} \in \mathbb{C}^{(n-r)\times (n-r)}$ with the eigenvalues $\lambda _{r+1}(A),\dots , \lambda _{n}(A)$, and $L_{12}\in \mathbb{C}^{r\times (n-r)}$. Note that $\bar{U} \in \mathcal{U} = \{ S \begin{bsmallmatrix} F_{1} \\ O_{n-r,r}\end{bsmallmatrix} \in \mathrm{St}(r,n) \ | \  F_{1}\in \mathbb{C}^{r\times r},\  F_{1}^{\dagger}F_{1}=I_{r} \} $. Using the change of the coordinate system by $S$,
\begin{align*}
& S^{\dagger} (A -\tilde{P}_{r} C^{\top}{(HH^{\top})}^{-1}C )  S
\\ =& 
S^{\dagger} AS - S^{\dagger}\bar{U} \tilde{R}_{\bar{U}} C_{\bar{U}}^{\top}{(HH^{\top})}^{-1}C_{\bar{U}} \bar{U} ^{\top}  S
\\ & 
- S^{\dagger} \bar{U} \tilde{R}_{\bar{U}} C_{\bar{U}}^{\top}{(HH^{\top})}^{-1}C(I_{n} - \bar{U}\bar{U} ^{\top})  S
\\ = &
\begin{bmatrix}
X & Y \\ O_{r,n-r} & L_{22}
\end{bmatrix}
\end{align*}
holds, where $X:=L_{11} - F_{1} \tilde{R}_{\bar{U}} C_{\bar{U}}^{\top}{(HH^{\top})}^{-1}C_{\bar{U}} F_{1}^{\dagger} \in \mathbb{C}^{r\times r}$, $Y \in \mathbb{C}^{r\times (n-r)}$, and $F_{1} \in \mathbb{C}^{r\times r}$ is a unitary matrix.  Since $A_{\bar{U}} = \bar{U}^{\top}A\bar{U} = F_{1}^{\dagger} L_{11} F_{1}$, $X = F_{1} (A_{\bar{U}}-\tilde{R}_{\bar{U}}C_{\bar{U}}^{\top}{(HH^{\top})}^{-1}C_{\bar{U}}) F_{1}^{\dagger}$. Because the unitary transformation preserves the eigenvalues, $X$ has the eigenvalues $\sigma _{1},\dots ,\sigma _{r}$, implying $A -\tilde{P}_{r} C^{\top}{(HH^{\top})}^{-1}C $ has the eigenvalues $\sigma _{1},\dots ,\sigma _{r}$ and $\lambda _{r+1}(A),\dots \lambda _{n}(A)$. 
}
\end{proof}

From Proposition \ref{prop:lrkf_bound}, we derive the condition for the necessary and sufficient rank $r$ for the bounded estimation error using the LRKB filter.

\begin{thm}\label{thm:lrkf_rankcondi}
Suppose that system $(A, G, C)$ is controllable and observable, and $A$ has $r^{\prime} $ unstable eigenvalues. Then, $A-\tilde{P}_{r}C^{\top}{(HH^{\top})}^{-1}C$ is stable iff $r\ge r^{\prime}$.  Furthermore, the error covariance matrix $\tilde{V}(t)$ is bounded and converges to a steady-state solution if $A-\tilde{P}_{r}C^{\top}{(HH^{\top})}^{-1}C$ is stable.
\end{thm}
\begin{proof}
Based on Propositions \ref{prop1}, \ref{prop:lrde_observe}, and \ref{prop:lrde_reach},
$\sigma_1,\dots,\sigma_r$ are stable eigenvalues. Thus, from Prop. \ref{prop:lrkf_bound}, if we select $r \geq r^{\prime}$, then $A-\tilde{P}_{r}C^{\top}{(HH^{\top})}^{-1}C$ is stable. Conversely, if $r<r^{\prime}$, then $A-\tilde{P}_{r}C^{\top}{(HH^{\top})}^{-1}C$ is unstable. 

According to the standard argument, the solution $\tilde{R}(t)$ of Eq. \eqref{eq_tildeR} is bounded, which directly implies the boundedness of $\tilde{P}(t)=\bar{U}\tilde{R}(t) \bar{U}^{\top} \in \mathbb{R}^{n\times n}$.  
Since $\tilde{V}(t)$ indicates the solution of Eq. \eqref{eq_tildeV} with time-varying, bounded coefficient matrices, $\tilde{V}(t)$ is a continuous solution. 
Furthermore, as $\lim _{t\to \infty}P(t) = P_{r}$, $A-\tilde{P}(t)C^{\top}{(HH^{\top})}^{-1}C$ becomes stable, which signifies that $\tilde{V}(t)$ converges to a positive semidefinite matrix. As observed in the continuity of $\tilde{V}(t)$ in time $t\geq 0$, $\sup _{t\geq 0}\tilde{V}(t)$ is bounded.  
\end{proof}

From Theorems \ref{thm:equili_localconv} and \ref{thm:lrkf_rankcondi}, we conclude that selecting $r$ larger than the number of unstable eigenvalues of $A$ is crucial for achieving bounded estimation. The number of unstable eigenvalues can be checked, for example, using the Routh table. However, finding a suitable rank for linear time-varying or nonlinear systems remains a challenge.

\section{Conclusion}\label{sec:Conclu}
In this study, we analyzed the properties of the Oja flow and a modified LRKB filter with a general matrix. We removed the restriction on symmetric matrices and investigated the Oja flow for general matrices. Based on the results, we revealed the rank $r$ condition for stable estimation. 

This paper focused on continuous-time LTI systems. In the future, the applicability of the modified LRKB filter with bounded estimation errors should be analyzed for discrete-time LTI, linear time-varying, and nonlinear systems.

\addtolength{\textheight}{-12cm}   % This command serves to balance the column lengths

%%%%%%%%%%%%%%%%%%%%%%%%%%%%%%%%%%%%%%%%%%%%%%%%%%%%%%%%%%%%%%%%%%%%%%%%%%%%%%%%

%%%%%%%%%%%%%%%%%%%%%%%%%%%%%%%%%%%%%%%%%%%%%%%%%%%%%%%%%%%%%%%%%%%%%%%%%%%%%%%%

%%%%%%%%%%%%%%%%%%%%%%%%%%%%%%%%%%%%%%%%%%%%%%%%%%%%%%%%%%%%%%%%%%%%%%%%%%%%%%%%

%\bibliographystyle{ieeetr}
%\bibliography{reference.bib}

\end{document}